\documentclass[a4paper,11pt]{amsart}
\usepackage[T1]{fontenc}
\usepackage{amsmath}
\usepackage{amsthm}
\usepackage{amssymb}
\usepackage{mathrsfs}
\usepackage[margin=1.2in]{geometry}
\usepackage{graphicx}
\usepackage[pdftex,%
bookmarksopen=false,
pdfpagemode=none % don't start with the navigation pane open in Acrobat
]{hyperref}

\DeclareSymbolFont{sfoperators}{OT1}{cmss}{m}{n}
\DeclareSymbolFontAlphabet{\mathsf}{sfoperators}

\makeatletter
\def\operator@font{\mathgroup\symsfoperators}
\makeatother

\theoremstyle{plain}
\newtheorem{theorem}{Theorem}[section]
\newtheorem{corollary}[theorem]{Corollary}

\newtheorem{lemma}[theorem]{Lemma}

\newtheorem*{claim}{Claim}
\newtheorem{conjecture}[theorem]{Conjecture}
\newtheorem{problem}[theorem]{Problem}
\theoremstyle{definition}
\newtheorem{definition}[theorem]{Definition}

\theoremstyle{remark}
\newtheorem*{remark}{Remark}

\newcommand{\h}{\mathscr{H}}
\DeclareMathOperator{\cone}{Cone}
\DeclareMathOperator{\lattice}{Lattice}
\DeclareMathOperator{\intcone}{IntCone}

\newenvironment{claimenv}{\list{}{\rightmargin0pt\leftmargin10pt\topsep0pt}\item[]}{\endlist}
\newenvironment{subproof}{\begin{claimenv}\begin{proof}}{\end{proof}\end{claimenv}}

\title{On Hilbert Bases of Cuts}
\author{Luis Goddyn}
\address[Luis Goddyn]{Department of Mathematics, Simon Fraser
  University, Burnaby, Canada}
\email{goddyn@sfu.ca}

\author{Tony Huynh}
\address[Tony Huynh]{Department of Computer Science, University of Rome, Rome, Italy}
\email{tony.bourbaki@gmail.com}

\author{Tanmay Deshpande}
\address[Tanmay Deshpande]{}
\email{tanmaydesh5886@gmail.com}

\thanks{This research was partially funded by an NSERC discovery grant. \\
Tony Huynh is supported by the European Research Council under the European Unions Seventh Framework
Programme (FP7/2007-2013)/ERC Grant Agreement no. 279558.}

\begin{document}

\begin{abstract}
A \emph{Hilbert basis} is a set of vectors $X \subseteq \mathbb{R}^d$ such that the integer cone 
%L
(semigroup) generated by $X$ is the intersection of the lattice generated by $X$ with the cone generated by $X$.  
Let $\h$ be the class of graphs whose set of cuts is a Hilbert basis in $\mathbb{R}^E$ (regarded as $\{0,1\}$-characteristic vectors indexed by edges).  
%L Somewhat surprisingly, w
We show that $\mathscr{H}$ is not closed under edge deletions, subdivisions, nor $2$-sums.  
%L We also show that  $K_6 \setminus e \notin \h$ and all graphs with $K_6 \setminus e$ as a minor do not belong to $\h$.  
Furthermore, no graph having $K_6 \setminus e$ as a minor belongs to $\h$.
This corrects an error in [M. Laurent. Hilbert bases of cuts. \textit{Discrete Math.}, 150(1-3):257-279 \textbf{(1996)}].

For positive results, we give conditions under which 
%L performing 2-sums does yield a graph in $\h$.  
the 2-sum of two graphs produces a member of $\h$.
Using these conditions we show that
all $K_5^{\perp}$-minor-free graphs are in $\h$, where $K_5^{\perp}$ is the unique 3-connected graph obtained by uncontracting an edge of $K_5$.  We also establish a relationship between 
edge deletion and subdivision. 
 Namely,  if $G'$  is  obtained from $G\in \h$ by subdividing $e$ two or more times,  then $G \setminus e \in \h$ if and only if $G'  \in \h$.
\end{abstract}

\maketitle

\section{Introduction}

Let $X$ be a set of vectors in $\mathbb{R}^d$.  We define 
\begin{align*}
\cone(X)&:=\left\{ \sum_{s \in X} c_s s : c_s \in \mathbb{R}_{\geq 0} \right\}, \\
\lattice(X)&:=\left\{ \sum_{s \in X} c_s s : c_s \in \mathbb{Z} \right\}, \\
\intcone(X)&:=\left\{ \sum_{s \in X} c_s s : c_s \in \mathbb{Z}_{\geq 0} \right\}.
\end{align*}

\begin{definition}
A set of vectors $X$ in $\mathbb{R}^d$ is a \emph{Hilbert basis} if 
\[
\intcone(X)=\cone(X) \cap \lattice(X).  
\]
\end{definition}

Hilbert bases 
%L are important objects in combinatorial optimization.  They 
were introduced by Giles and Pulleyblank~\cite{tdi} as a tool to study total dual integrality.  They are also connected to set packing, toric ideals and perfect graphs \cite{OSheaSebo}.   Combinatorially defined Hilbert bases have computational consequences, since membership testing is often easier for the cone and the lattice than for the integer cone.  This is the case, for example, with 
%L the edge colourings 
edge colouring and the set of perfect matchings of 
 a regular graph 
 \cite{EdmondsMatchingPolytope, LovaszMatchingLattice}.   We are interested here in the class of finite graphs whose sets of edge cuts form Hilbert bases.

All graphs here are assumed to be finite. Let $G=(V,E)$ be a  graph. A \emph{circuit} is the edge set of a cycle of $G$.
For $S \subseteq V$, we denote by $\delta(S) = \delta_G(S)$ the set of edges in $G$ having exactly one endpoint in $S$, and call $\delta(S) $ the \textit{cut} in $G$ generated by $S$.
%A \emph{bond} of $G$ is an inclusionwise minimal nonempty cut of $G$.
Regarding each cut
%L $F := \delta(S)$ 
$\delta(S)$ as a $\{0,1\}$-characteristic vector 
%L $\chi^F \in \mathbb R^E$.
in $\mathbb R^E$, we define the vector set
% in $\mathbb{R}^E$, the set of real vectors indexed by $E$.
%L 
$\mathcal B(G): = \{\delta(S) : S \subseteq V(G)\}$.
We define $\h$ to be the class of finite graphs $G$ for which $\mathcal B(G)$ forms a Hilbert basis.

Our aim is to study the class $\h$.  We remark that the 
%L
(matroidal) version of the dual problem was completely solved by Alspach, Goddyn, and Zhang~\cite{graphswithccp}, where they 
show that the set of circuits of a graph $G$ is a Hilbert basis if and only if $G$ does not contain the Peterson graph as a minor.  
In contrast, the class $\h$ is less well-behaved.  For example, we show $\h$ is not closed under edge deletions, subdivisions, nor 2-sums.  
%L We also show that $K_6 \setminus e \notin \h$ and all graphs having $K_6 \setminus e$ as a minor do not belong to $\h$.  
Furthermore, we show that no graph having $K_6 \setminus e$ as a minor belongs to $\h$.
This corrects an error in Laurent~\cite{laurent}.  

For positive results, we give conditions under which performing 2-sums does yield a graph in $\h$, and use this to show that
all $K_5^{\perp}$-minor-free graphs are in $\h$, where $K_5^{\perp}$ is the unique 3-connected single-element uncontraction of $K_5$.  
%L 

We also establish a relationship between 
edge deletion and subdivision;
%L  for graphs in $\h$.
if $G \in \h$ and $G'$  is obtained from $G$ by subdividing an edge $e$ two or more times, then $G \setminus e \in \h$ if and only if $G' \in \h$.

%Finally, we use some computational results of Deshpande~\cite{tanmay} to show
%that all uncontractions of $K_5$ are in $\h$.  

\section{Previous Results}
In this section, we review some previous results that we will need later.  

A \emph{bond} of a graph $G$ is an inclusionwise minimal nonempty cut of $G$.  Every cut of $G$ is a disjoint union of bonds, so
the bonds generate the same lattice and cone as $\mathcal B(G)$.
Indeed, there is a bijection between the bonds of $G$ and the extreme rays of $\cone(\mathcal B(G))$.
Gordan \cite{Gordan} (see \cite{Schrijver}) showed that for every finite set of vectors $X$ such that $\cone(X)$ has a vertex, there is a unique minimal set of vectors $\overline{X}$ which generate the same cone and lattice as $X$, such that $\overline{X}$ is a Hilbert basis.  The set $\overline{X}$ is called the \emph{minimal Hilbert basis} for $X$.  The vectors in $\overline{X} \setminus X$ are called \emph{quasi-Hilbert elements}.
  We shall give specific examples of quasi-Hilbert elements in the next section.
  
The lattice generated by the cuts of a graph is characterized as a special case of a general statement \cite[Proposition 2.4]{matchings} regarding the cocircuits of  a binary matroid having no Fano-minor.

\begin{lemma} \label{LOC}
 For every simple graph $G$ and $x\in \mathbb{Z}^{E(G)}$, $x \in \lattice(\mathcal B(G))$ if and only if  $\sum_{e\in C} x_e$ is even for each circuit $C$ of $G$.
\end{lemma}

We use the notation  $x(C) := \sum_{e\in C} x_e$. 
The cone generated by $\mathcal B (G)$ is very complicated in general.  
%For example, the membership problem is NP-hard if $G$ is input, and there is no compact description of the facets $\cone(\mathcal B(G))$. 
See \cite{AvisImaiIto, cutconefacets} for a discussion of this.  
Seymour~\cite{multiflows} characterized those graphs (and matroids) for which  $\cone(\mathcal B (G))$ is described by a natural family of inequalities called \emph{cycle constraints}.

\begin{lemma} \label{noK5}
For every  graph $G$ with no $K_5$-minor and every $x \in \mathbb{R}^{E(G)}$, $x \in \cone(\mathcal B(G))$ if and only if $x_e \geq 0$ for all $e \in E(G)$; and
\begin{equation*}
x_e \leq x (C \setminus \{e\}),
\end{equation*}
for all circuits $C$ of $G$ and all edges $e \in C$.
\end{lemma}

Using this result, Fu and Goddyn~\cite{circuitcover} showed that every $K_5$-minor-free graph is in $\h$.

\begin{lemma} \label{noK5in}
All $K_5$-minor-free graphs are in $\h$.  
\end{lemma}

Earlier, Deza~\cite{deza1}~\cite{deza2} had shown that $K_5 \in \h$ as well.

 \begin{lemma} \label{K5in}
$K_5 \in \h$.
\end{lemma}

In fact, we now give an explicit description of $\cone(\mathcal{B}(K_5))$.  

\begin{lemma}[\cite{laurent}] \label{hyper}
Let $G=K_5$ with $V(G)=\{1, \dots, 5\}$. 
%L
The facets of $\cone(\mathcal B (K_5))$ correspond to the following 50 inequalities.
\begin{itemize}
\item
$x_e \geq 0$ for all $e \in E(G)$,  
\item
$x_e \leq x (C \setminus \{e\})$ for every 3-circuit $C$ in $G$ and all $e \in C$, and
\item
$\sum_{1 \leq i < j \leq 5} b_ib_jx_{ij} \leq 0$, for all ten permutations $b$ of the vector $(1,1,1,-1,-1)$.   
\end{itemize}
\end{lemma}

%The above constraints  are called \emph{cycle constraints}.  We remark that it is a difficult problem to describe $\cone(\mathcal B (G))$ in general~\cite{cutconefacets}.  However, it is easy to see that all cycle constraints are valid for $\cone(\mathcal B (G))$ for arbitrary $G$.  

%The next lemma asserts that $\h$ is closed under edge contractions.  

The third type of constraint in Lemma~\ref{hyper} is called a \emph{hypermetric inequality}.
The next lemma follows  from the observation that $\cone(\mathcal B (G / e))$ is a face of $\cone(\mathcal B (G))$.
\begin{lemma}[\cite{laurent}] \label{contract}
If $G \in \h$, then $G / e \in \h$ for all $e \in E(G)$.  
\end{lemma}

%Indeed, it is fairly easy to check that $\cone(\mathcal B (G / e))$ is a face of $\cone(\mathcal B (G))$, from which the result follows easily.  

%L We also require a lemma telling us how we can combine graphs in $\h$ to obtain a new graph in $\h$.  Toward that end,
Another lemma of Laurent will also be used later.
Let $G_1$ and $G_2$ be two graphs both containing a clique $K$ with $n$ vertices.  The \emph{$n$-clique sum (along $K$)} of $G_1$ and $G_2$ is the graph obtained by gluing $G_1$ and $G_2$ along $K$ and keeping the set of edges of exactly one copy of $K$.  On the other hand, the \emph{$n$-sum (along $K$)} of $G_1$ and $G_2$ is the graph obtained by gluing $G_1$ and $G_2$ along $K$ and deleting the edges from both copies of $K$.  

We denote the $n$-clique sum and the $n$-sum of $G_1$ and $G_2$ (along $K$), as $G_1 \oplus_K G_2$ and $G_1 +_K G_2$ respectively.  We warn the reader that this notation is non-standard.  However, it is important for us to differentiate between $\oplus_K$ and $+_K$.  

%For classes of graphs that are closed under edge deletion, this distinction is inconsequential.  However, as we shall see, the class $\h$ is not closed under edge deletion.  
%Therefore, when applying structure theory results, it is essential to work with $+_K$, rather than $\oplus_K$.  

\begin{lemma}[\cite{laurent}] \label{kcliquesums}
Let $G_1$ and $G_2$ be graphs both containing a clique $K$ with at most 3 vertices.  If $G_1$ and $G_2$ are both in $\h$, then $G_1 \oplus_K G_2 \in \h$.
\end{lemma}

\section{Negative Results}

In this section
%L  we prove some negative results concerning membership in $\h$.   This requires us to 
we exhibit some graphs which are not in $\h$.  
For graphs with fewer than about 15 edges, membership in $\h$
%L Such graphs were found with the aid of the software package 
can be (and were) tested with the aid of software such as 
%\emph{Normaliz}~\cite{normaliz1}~\cite{normaliz2},
 \emph{Normaliz}~\cite{normaliz3, normaliz4} which can recognize Hilbert bases, and compute quasi-Hilbert elements.  
 However, we stress that no proofs in this section rely on such `black box' computations.  
%Indeed, for all graphs $G \notin \h$ in this section, we clearly show by hand why $G \notin \h$. 
%We warn that some calculations are slightly tedious, so the trusting reader can skip them, or independently verify them by computer. 
%Before starting, we will need a few definitions and lemmas.  Let $G$ be a graph.  We let $\cone(G), \lattice(G)$, and $\intcone(G)$ denote the cone, lattice, and integer cone generated by the set of cuts of $G$.  
% LATER? We let $K_n$ be the complete graph on $n$ vertices and $K_6 \setminus e$ be the graph obtained from $K_6$ by deleting an edge $e$. 
%The first lemma we need is that $\lattice(G)$ has a very simple description.  
%The lattice generated by $\mathcal B(G)$ is easily characterized. 
%

In~\cite{laurent}, it is shown that $K_6 \notin \h$, but there is an incorrect claim that all proper subgraphs of $K_6$ are in $\h$. 
%L The error in \cite{laurent} arises from equation (10) of page 270, which incorrectly assumes $e_{56} \in E(G_6)$.
%L We now show that in fact, $K_6 \setminus e \notin \h$.  

\begin{theorem} \label{k6minusedge}
 $K_6 \setminus e \notin \mathscr{H}$. Furthermore, if $G$ contains a $(K_6 \setminus e)$-minor, then $G \notin \mathscr{H}$.
\end{theorem}

\begin{proof}
To show that $K_6 \setminus e \notin \h$, we 
%L must 
exhibit a vector that is in $\cone(\mathcal B (K_6 \setminus e)) \cap \lattice(\mathcal B(K_6 \setminus e))$, but not in 
$\intcone(\mathcal B (K_6 \setminus e))$. 
We show such a vector $x$ in Figure~\ref{k6minuse}.  
\begin{figure}[htbp]
	\centering
	\includegraphics[keepaspectratio=true,width=3in]{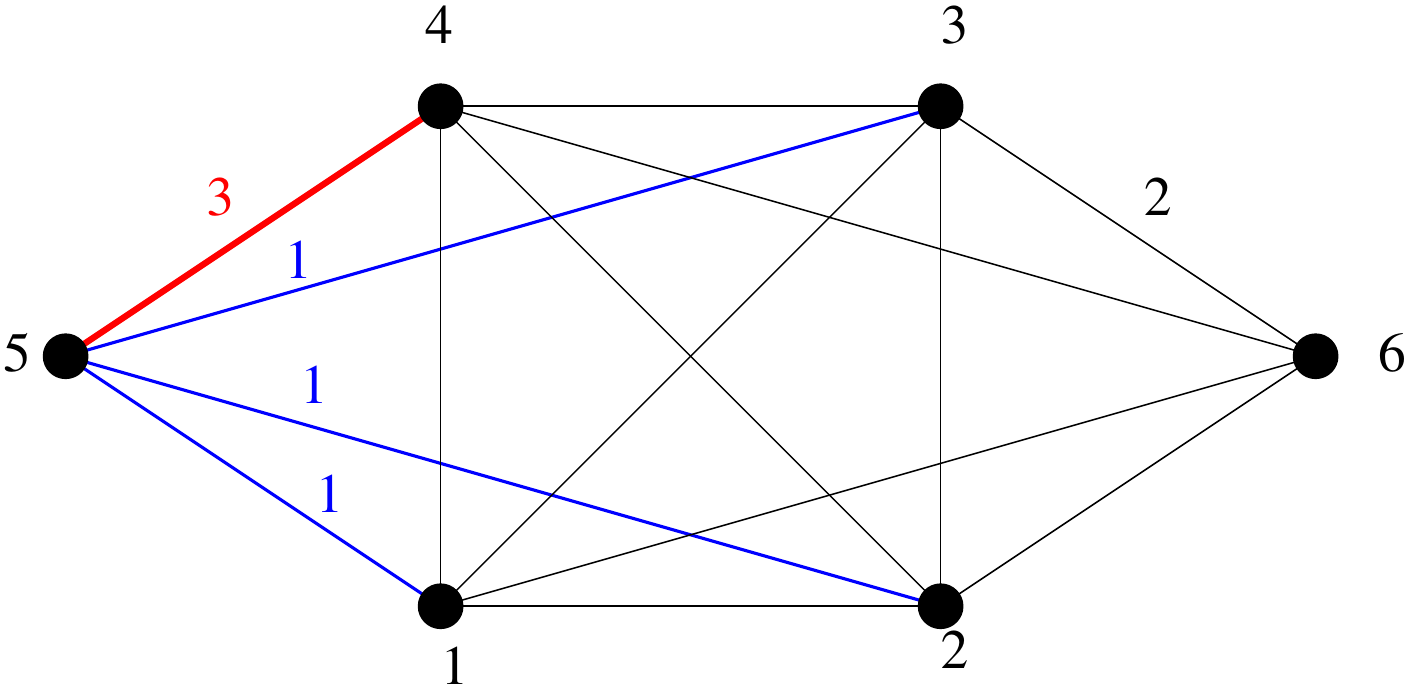} 
	\caption[Unlabelled edges all have weight 2.]{Unlabelled edges all have weight 2.}
	\label{k6minuse}
\end{figure}
The fact that $x \in \lattice(\mathcal B(K_6 \setminus e))$ follows from Lemma~\ref{LOC}. 
%L 
To see that  $x \in \cone(\mathcal B (G))$,
we observe that 
%%L $x = \frac{1}{2}\sum_S (\delta(S))$ 
%$x = \frac{1}{2}\sum_S \delta(S)$ where $S$ ranges over 
%\begin{equation} 
%\label{eq:coneSupport}
%     \{ 1,4 \} , \{ 2,4 \}, \{ 3,4 \}, \{ 1,4,6 \}, \{ 2,4,6 \}, \{ 3,4,6 \}, \{ 6 \}.
%\end{equation}
%%L 
%%Suppose that $x \in \intcone(\mathcal B (K_6 \setminus e))$. 
%%%LFinally, the vector $x$ lies
%%We first check that $x$ lies on exactly seven facets of $\cone(\mathcal B (K_6 \setminus e))$. 
%%These facets correspond to one cycle inequality for each of the triangles $\{ 1,4,5\}$, $\{2,4,5\}, \{3,4,5\}, \{1,2,5\}, \{2,3,5\}, \{1,3,5\}$, and the hypermetric inequality
%%\[
%%x_{12} + x_{23} + x_{13} + x_{45} - \sum_{j=4}^{5} \sum_{i=1}^{3} x_{ij} \leq 0.
%%\]
%%
%%We can verify by hand which cuts are tight for all of the above facets. These are precisely the cuts $\delta(S)$
%%%L  where $S$  ranges over \[\{ 1,4 \} , \{ 2,4 \}, \{ 3,4 \}, \{ 1,4,6 \}, \{ 2,4,6 \}, \{ 3,4,6 \}, \{ 6 \}.\]
%%where $S$ is listed in \eqref{eq:coneSupport}.
%One easily verifies that $x - \delta(S) \notin \cone(\mathcal B (K_6 \setminus e))$ for all such $S$. Therefore, $x \notin \intcone(\mathcal B (K_6 \setminus e))$. 
$x = \sum \mathcal S$ where
\begin{equation} 
\label{eq:coneSupport}
    \mathcal S = \{ \delta(S) : S =  \{ 1,4 \} , \{ 2,4 \}, \{ 3,4 \}, \{ 1,4,6 \}, \{ 2,4,6 \}, \{ 3,4,6 \}, \{ 6 \} \; \}.
\end{equation}
Suppose, for a contradiction, that $x$ is a positive integer combination of a set $\mathcal T$ of cuts of $K_6-e$.
%We claim that $\mathcal T \subseteq \mathcal S$.
One easily checks that the seven cuts in $\mathcal S$ are linearly independent (for example, consider their intersections with the edge sets
$\{ i5 : i=1,2,3 \}$ and $\{ i6 : i=1,2,3 \}$).
Therefore the cuts generate a 7-dimensional subcone $\mathcal{K} \subseteq \cone(\mathcal B (G))$
% which is a face of dimension $7$ and
with codimension $14-7=7$.
We now verify that $x$ lies in the intersection of seven linearly-independent facet-defining inequalities for $\cone(\mathcal B (G))$.
These facets correspond to a cycle inequality for each of the triangles $\{ 1,4,5\}$, $\{2,4,5\}, \{3,4,5\}, \{1,2,5\}, \{2,3,5\}, \{1,3,5\}$, and the hypermetric inequality
of Lemma \ref{hyper} (applied to $K_6 -6$) with $b=(1,1,1,-1,-1)$.
%\[
%x_{12} + x_{23} + x_{13} + x_{45} - \sum_{j=4}^{5} \sum_{i=1}^{3} x_{ij} \leq 0.
%\]
Therefore $\mathcal{K}$ is a facet of $\cone(\mathcal B (G))$. It follows that $\mathcal T \subseteq \mathcal{K}$.
But we have that $\mathcal B(G) \cap \mathcal{K} = \mathcal S$, so $\mathcal T \subseteq \mathcal S$.
Thus $x$ is a non-negative integer combination of vectors in $S$.
Since every cut in $\mathcal S$ has size 4 or 8, the sum of the entries in $x$ must be a multiple of $4$.
However the sum of the entries of $x$ is $26 \equiv 2 \pmod 4$.
This contradiction proves that $x \notin \intcone(\mathcal B(G))$.

For the second part, suppose that $G / C \setminus D \cong K_6 \setminus e$.  Let $G_1:=G / C$, and $G_2$ be the graph obtained from $G_1$ be removing loops, parallel edges, and isolated vertices.  
Observe that  $G_1 \in \h$ if and only if $G_2 \in \h$.  However, by Lemma~\ref{contract}, if $G \in \h$, then $G_1 \in \h$, so $G_2 \in \h$.  But, $G_2$ is either $K_6$ or $K_6 \setminus e$, both of which are not in $\h$.  
\end{proof}

\begin{remark}
The exact error in~\cite{laurent} occurs in equation (10) on page 270 where $(5,6)$ is erroneously assumed to be an edge of $K_6 \setminus (5,6)$.
\end{remark}

Let $K_5^{\perp}$ be the graph in Figure~\ref{h6}, with a distinguished edge $e$.

\begin{figure}[htbp]
	\centering
	\includegraphics[keepaspectratio=true,width=2.5in]{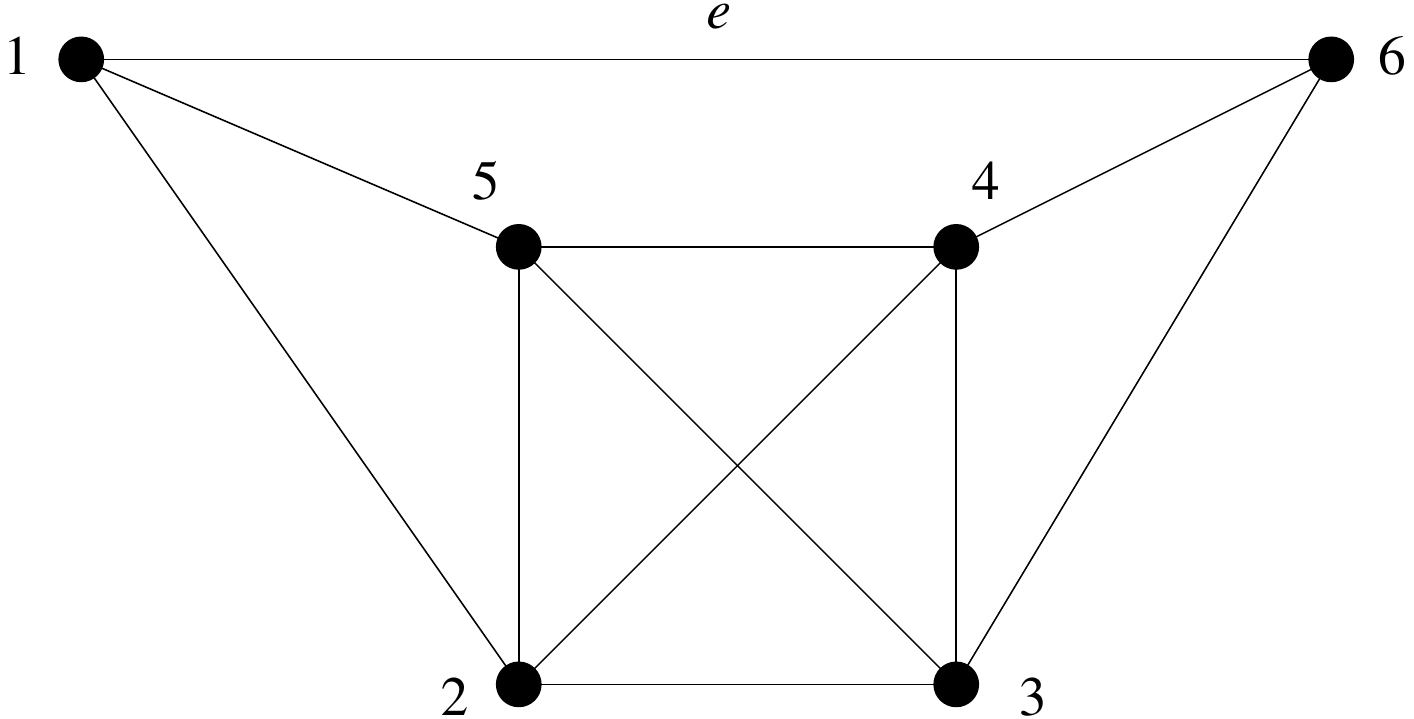}
	\caption[The graph $K_5^{\perp}$]{The graph $K_5^{\perp}$.}
	\label{h6}
\end{figure}

It is shown in~\cite{laurent} that
$K_5^{\perp} \in \h$.
On the other hand, we claim that the 2-sum $H_{10}^- := K_5^{\perp} +_e K_5^{\perp}$ 
%performing a 2-sum of $K_5^{\perp}$ with itself along $e$ yields a graph 
is not in $\h$. 

\begin{lemma}\label{2h6}
%$K_5^{\perp} +_e K_5^{\perp} \notin \h$.  
$H_{10}^- \notin\h$.
\end{lemma}

\begin{proof}
\begin{figure}[htbp]
	\centering
	\includegraphics[keepaspectratio=true,width=5.5in]{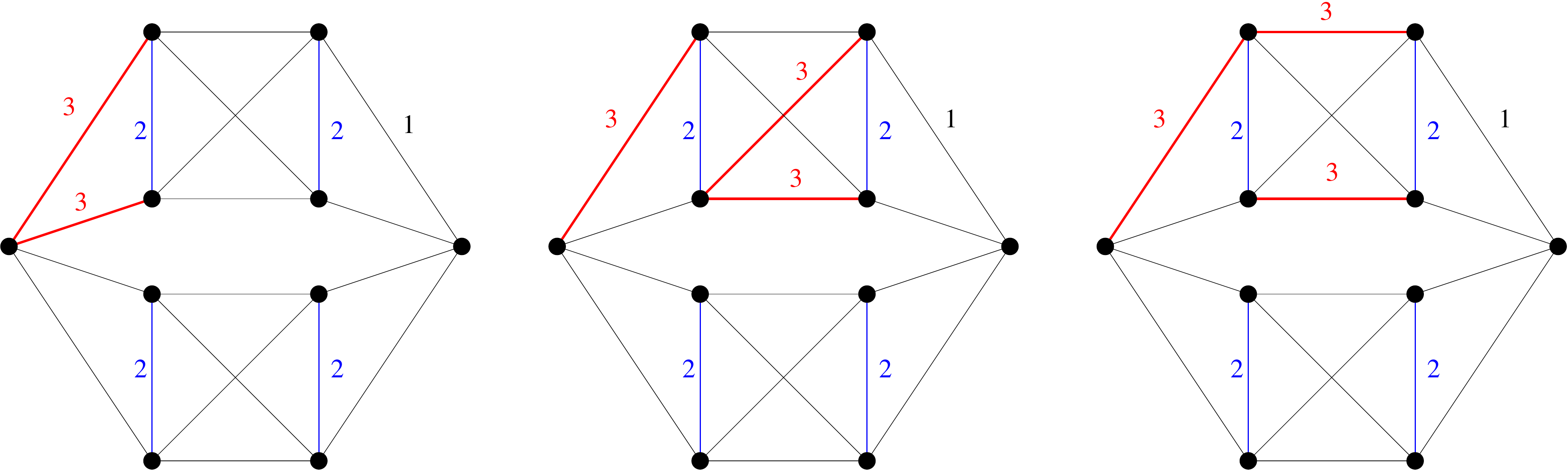} 
	\caption[Unlabelled edges all have weight 1.]{Unlabelled edges all have weight 1.}
	\label{h10}
\end{figure}

Let $x$ be the first vector given in Figure~\ref{h10}.  We show that $x \in \cone(\mathcal{B}(H_{10}^-)) \cap \lattice(\mathcal{B}(H_{10}^-))$, but $x \notin \intcone(\mathcal{B}(H_{10}^-))$.  First, $x = \frac{1}{2}\sum_S (\delta(S))$ where $S$ ranges over 
\[
\{ 1,4 \} , \{ 2,3 \}, \{ 1,3 \}, \{ 2,4 \}, \{ 5,7,9 \}, \{ 5,8,10 \}, \{ 5,7,10 \}, \{5,8,9 \},
\]
with vertices labelled as in Figure~\ref{facet1}. Thus, $x \in \cone(\mathcal B (H_{10}^-))$. It is also clear that $x \in \lattice(\mathcal B (H_{10}^-))$ by Lemma~\ref{LOC}.

\begin{figure}[htbp]
	\centering
	\includegraphics[keepaspectratio=true,width=1.8in]{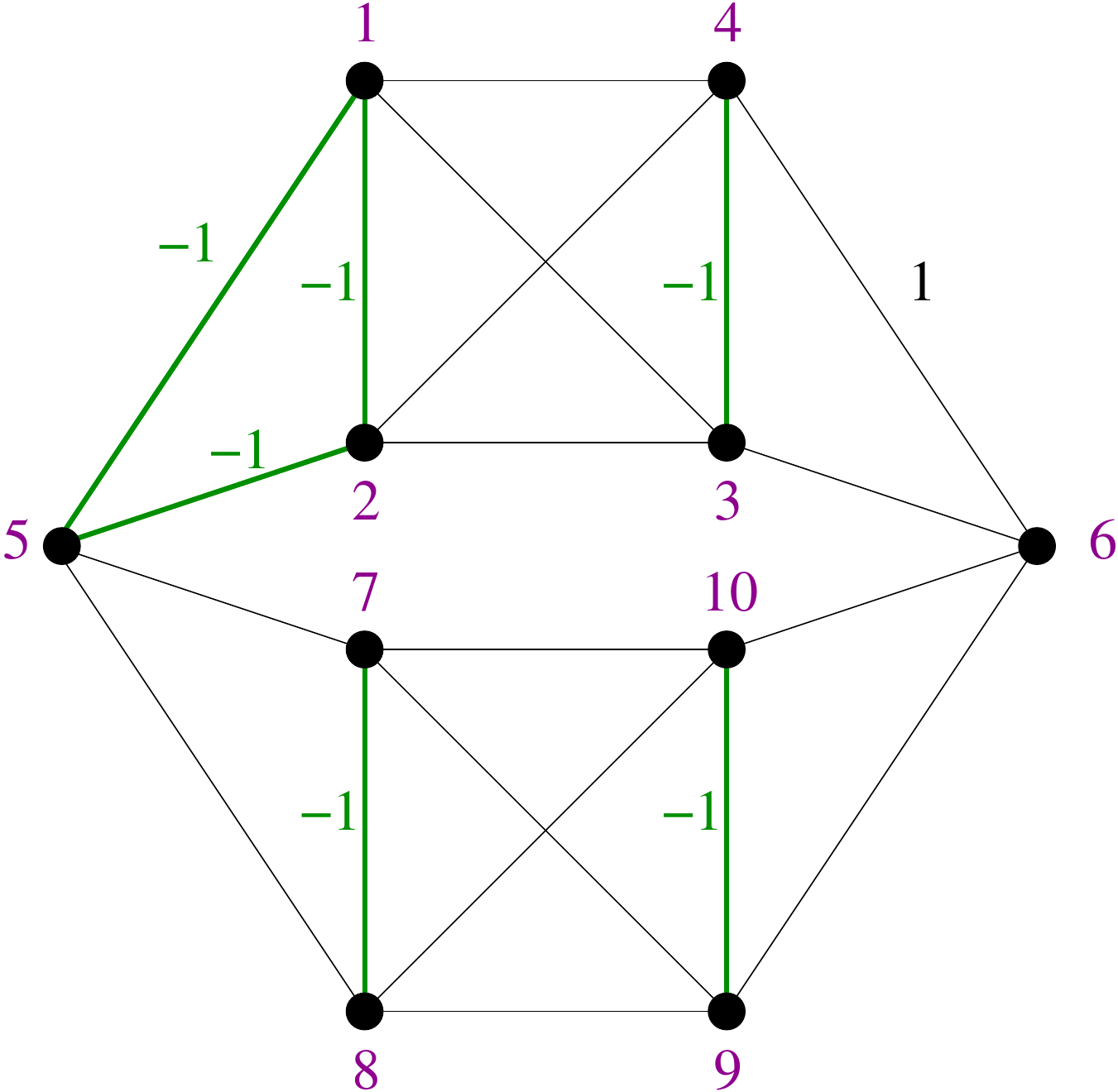}
	\caption[Unlabelled edges all have weight 1.]{The constraint vector $v$ for a facet of $H_{10}^-$. All unlabelled edges have weight 1.}
	\label{facet1}
\end{figure}

Now $x$ lies on the following facets of $\cone(\mathcal B (H_{10}^-))$. These are the cycle inequalities determined by the triangles $\{ 1,2,3 \} , \{ 1,2,4 \}, \{ 1,3,4 \}, \{ 2,3,4 \}, \{ 3,4,6 \}, \{ 5,7,8 \}, \{ 7,8,9 \}$, $\{7,8,10 \}, \{ 7,9,10\}, \{8,9,10\}, \{6,9,10\}$ and the inequality $\sum_{e\in E} v_e x_e \le 0$ where the constraint vector $v$ is  shown in Figure~\ref{facet1}.

We can check by hand which cuts are tight for all of the above facets. These are precisely the cuts $\delta(S)$ where $S$ ranges over 
\[
\{ 1,4 \} , \{ 2,3 \}, \{ 1,3 \}, \{ 2,4 \}, \{ 5,7,9 \}, \{ 5,8,10 \}, \{ 5,7,10 \}, \{5,8,9 \}.
\]
One easily checks that $x - \delta(S) \notin \cone(\mathcal B (H_{10}^-))$ for all of the above $S$. Therefore,
$x \notin \intcone(\mathcal B (H_{10}^-))$, as required.  
\end{proof}

We thus have the following corollary. 

\begin{corollary}\label{2sum}
The class $\mathscr{H}$ is not closed under 2-sums.
\end{corollary}

We now show that $\h$ is also not closed under edge deletions or subdivisions.  
Let $H_{10} := K_5^{\perp} \oplus_e K_5^{\perp} $,  so that $H_{10}^- = H_{10} \setminus e$.
%be graph obtained by taking the 2-clique sum of $K_5^{\perp}$ with itself along the distinguished edge $e$.  
%We still regard $e$ as an edge of $H_{10}$, and 
Let $H_{11}$ be the graph obtained from $H_{10}$ by subdividing $e$ once.  

\begin{lemma}
$H_{10} \in \h$, but $H_{11} \notin \h$. 
\end{lemma}

\begin{proof}
The fact that $H_{10} \in \h$ follows from Lemma~\ref{kcliquesums}.  
\begin{figure}[htbp]
	\centering
	\includegraphics[keepaspectratio=true,width=5.5in]{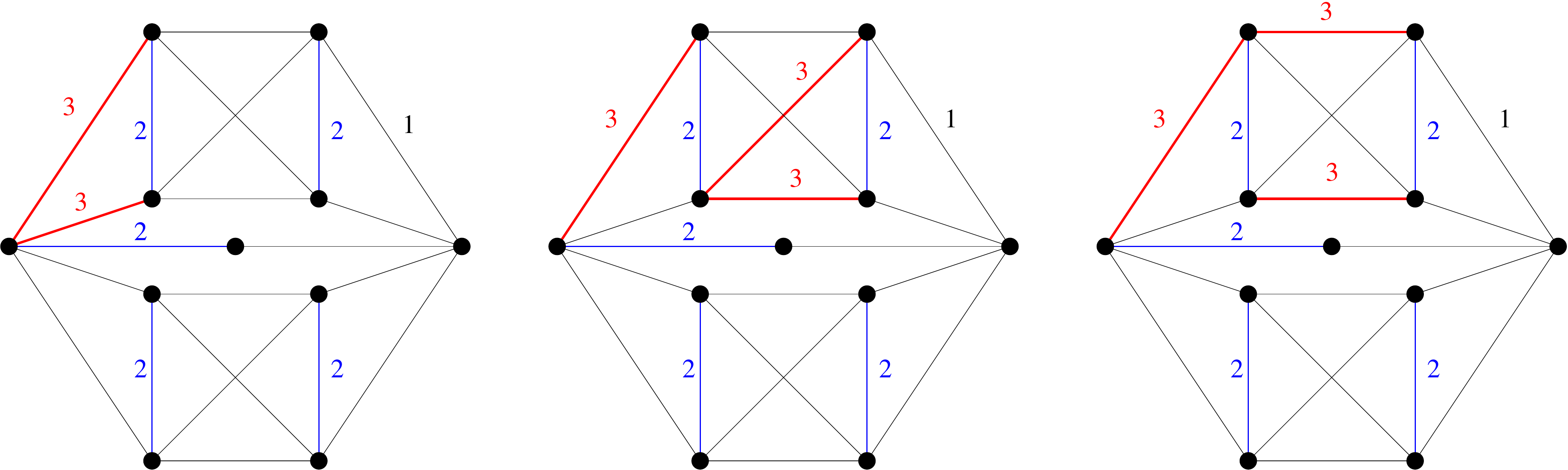} 
	\caption[Unlabelled edges all have weight 1.]{Unlabelled edges all have weight 1.}
	\label{h11}
\end{figure}
For the second part, let $y$ be the first vector given in Figure~\ref{h11}. Let the vertices of $H_{11}$ be labelled as in Figure~\ref{facet1} with the additional vertex labeled $11$. Then $y = \frac{1}{2}\sum_S (\delta(S))$ where $S$ ranges over 
\[
\{ 1,4 \} , \{ 2,3 \}, \{ 1,3 \}, \{ 2,4 \}, \{ 5,7,9 \}, \{ 5,8,10,11 \}, \{ 5,7,10,11 \}, \{5,8,9,11 \}, \{11\}.
\]
Thus, $y \in \cone(\mathcal B (H_{11}))$. It is also clear that $y \in \lattice(\mathcal B (H_{11}))$ by Lemma~\ref{LOC}.

%Let $f$ and $g$ be the edges incident to the new vertex $11$ and let $y'$ be the restriction of $y$ to $E(H_{10}^-) = E(H_{11}) \setminus \{f,g\}$.   
% Note that $H_{11} \setminus \{f,g\} \cong K_5^{\perp} +_e K_5^{\perp}$. 
Consider the subgraph $H_{10}^- = H_{11} - 11$, and the restriction $y'$ of $y$ to $E(H_{10}^-)$. 
If $y \in \intcone (\mathcal B (H_{11}))$, then $y' \in \intcone (\mathcal B (H_{10}^-))$.  But $y'=x$, where $x$ is the vector from the proof of Lemma~\ref{2h6}, a contradiction.  
\end{proof}

Since $H_{10}^- = H_{10} \setminus e$, we obtain the following corollary.

\begin{corollary}\label{edel}
The class $\mathscr{H}$ is not closed under edge deletions or subdivisions.
\end{corollary}

\begin{remark}The vectors from Figure~\ref{h10} and Figure~\ref{h11} are actually the quasi-Hilbert elements of $\mathcal{B}(H_{10}^-)$ and $\mathcal{B}(H_{11})$, respectively.  However,  we will  not need (or show) this.
\end{remark}

\section{Positive Results}
%\subsection{Uncontractions of $K_5$ and $K_5^{\perp}$-free graphs}

By Corollary~\ref{2sum}, $\h$ is not closed under 2-sums.  On the other hand, in this section, we show that under
some additional assumptions, performing a 2-sum does yield a graph in $\h$.  We will also give two applications of this
theorem.  

Before starting, we require a few definitions and lemmas.

\begin{definition}
 Let $G$ be a graph with a fixed edge $f$. Let $x \in \mathbb{R}^{E(G)}$ and let $x(\gamma) \in \mathbb{R}^{E(G)}$ be the vector obtained from $x$ by changing the entry $x_f$ to $\gamma$. Define the \emph{feasibility interval} $I(G,x,f)$ for $G, x$ and $f$ to be the (possibly empty) interval $[\gamma_{\min}, \gamma_{\max}]$ such that $x(\gamma) \in \cone(\mathcal B (G))$ if and only if $\gamma \in [\gamma_{\min}, \gamma_{\max}]$.
\end{definition}

\begin{definition}
 Let $G$ be a graph with a fixed edge $f$. Define a vector $x \in \mathbb{R}^{E(G)}$ to be \textit{almost in the lattice of $G$ with respect to $f$} if $x$ restricted to $E(G) \setminus \{f\}$ is in $\lattice(\mathcal B (G \setminus f))$.
\end{definition}

\begin{lemma}\label{lemma1}
 Let $G$ be a graph with a fixed edge $f$. If $G \setminus f \in \mathscr{H}$, then for every $x \in \mathbb{R}^{E(G)} $ such that $x$ is almost in the lattice with respect to $f$ and $I(G,x,f) \neq \emptyset$, there exists $\gamma \in \mathbb{Z}_{\geq 0}$ such that $x(\gamma) \in \intcone(\mathcal B (G))$.  
\end{lemma}
\begin{proof}
Let $f$ be an edge of $G$ such that $G \setminus f \in \h$, $x \in \mathbb{R}^{E(G)} $ is almost in the lattice of $G$ with respect to $f$, and $I(G,x,f)$ is non-empty.  Let $G' := G \setminus f$ and $x' \in \mathbb{Z}^{E(G')}$ be the restriction of $x$ to $E(G')$.  Observe that $x' \in \cone(\mathcal B (G')) \cap \lattice(\mathcal B (G'))$. 

 Since $G' \in \mathscr{H}$, we also have $x' \in \intcone(\mathcal B (G'))$. Therefore, there exist $\alpha_S \in \mathbb{Z}_{\geq 0}$ such that 
\begin{equation*}
x' = \sum_{S \subseteq V(G)} \alpha_S \delta_{G'}(S).
\end{equation*}
Define $\mathcal{F} := \{S \subseteq V(G): f \in \delta_G(S)\}$ and set $\gamma := \sum_{S \in \mathcal{F}} \alpha_S$.  The above sum shows that $x(\gamma) \in \intcone(\mathcal B (G))$, as required.  
\end{proof}

\begin{definition}
 Let $G$ be a graph with a fixed edge $f$. We say that $G$ has the \emph{lattice endpoint property (with respect to $f$)} if for all $x \in \mathbb{R}^{E(G)}$ that is almost in the lattice with respect to $f$ and for each endpoint $\gamma$ of $I(G,x,f)$, we have $\gamma=0$ or $x(\gamma) \in \lattice(\mathcal B (G))$.  
\end{definition}

The next lemma will be useful for verifying the lattice endpoint property.

\begin{lemma} \label{evencircuit1}
Let $G$ be a graph, $f \in E(G)$, and $x \in \mathbb{R}^{E(G)}$ be almost in the lattice with respect to $f$.  If there is a circuit $C$ such that $f \in C$ and $x(C)$ is even, then $x \in \lattice(G)$.  
\end{lemma}

\begin{proof}
Suppose not.  Then there is a circuit $C'$ containing $f$ such that $x(C')$ is odd.  But now,
\[
x(C \Delta C')=x(C)+x(C')-2x(C \cap C') \equiv 1 \pmod 2.
\]
Thus $C \Delta C'$ contains a circuit $C''$ such that $x(C'')$ is odd.  Since $f \notin C''$, this contradicts the fact that $x \in \mathbb{R}^{E(G)}$ is almost in the lattice with respect to $f$,   
\end{proof}

\begin{lemma} \label{noK5LEP}
All simple $K_5$-minor-free graphs have the lattice endpoint property with respect to every edge.
\end{lemma}

\begin{proof}
Let $G$ be a simple $K_5$-minor-free graph, $f \in E(G)$, $x \in \mathbb{R}^{E(G)}$ be almost in the lattice with respect to $f$, and $\gamma$ be an endpoint of $I(G,x,f)$.  Since $\gamma$ is an endpoint, there is a constraint involving $x_f$ for which $x(\gamma)$ is tight for. By Lemma~\ref{noK5}, such a constraint must be a non-negativity constraint or a cycle constraint.  If it is a non-negatively constraint, then we have $\gamma=0$, as required.  So, we may assume that $x(\gamma)$ satisfies some cycle constraint with equality.  In particular, this implies that there is a circuit $C$ of $G$ such that $f \in C$ and $x(C)$ is even.  By Lemma~\ref{evencircuit1}, $x(\gamma) \in \lattice(G)$, as required.
\end{proof}

\begin{lemma} \label{k5LEP}
$K_5$ has the lattice endpoint property with respect to every edge.
\end{lemma}

\begin{proof}
Let $f \in E(K_5)$ and $x \in \mathbb{R}^{E(K_5)}$ be almost in the lattice with respect to $f$, and $\gamma$ be an endpoint of $I(K_5,x,f)$.  
As in Lemma \ref{noK5LEP}, if $x(\gamma)$ is tight 
for a non-negativity constraint or a cycle constraint we are done.  Thus, by Lemma~\ref{hyper} we may assume that $x(\gamma)$ is tight for a hypermetric inequality determined
by a permutation of the vector $b=(1,1,1,-1,-1)$.  In particular, $x(\gamma)$ is integer-valued and $x(\gamma)(E(K_5))$ is even.  
The edges of $K_5$ may be partitioned into two circuits $C, C'$
% decomposition of  $K_5$ into cycles, 
and we may assume $f \in C$. 
Since $x$ is almost in the lattice with respect to $f$,  we have that $x(\gamma)(C')$ is even.  
It follows that 
\[
x(\gamma)(C) = x(\gamma)(E(K_5)) - x(\gamma)(C') \equiv 0 - 0 \pmod 2.
\]
%Let the vertices of $K_5$ be $\{1,\dots, 5\}$ with $f=12$.  Consider the set of cycles $C_1=123, C_2=124, C_3=125$, and $C_4=345$.  Observe that 
%\[
%\sum_{i=1}^{4} x(C_i) \equiv x(E(K_5)) \pmod 2,
%\]
%since on the left-hand side $x_f$ appears three times and $x_e$ occurs once for all other edges. Since $x$ is almost in the lattice with respect to $f$, $x(C_4)$ is even.  
%As $x(E(K_5))$ is also even, it follows that for some $i \in \{1, 2, 3\}$, $x(C_i)$ is even.  
But now $x(\gamma) \in \lattice(G)$ by Lemma~\ref{evencircuit1}, and we are done.
\end{proof}

We can now state and prove the main result of this section.

\begin{theorem}\label{main}
Let $G_1$ and $G_2$ be 2-connected graphs with $E(G_1) \cap E(G_2) = f$. 
If  the four graphs $G_1$, $G_2$, $G_1 \setminus f$, $G_2 \setminus f$ are in $\mathscr{H}$, 
%$G_1 \setminus f$ and $G_2 \setminus f$ are both connected, 
and $G_1$ has the lattice endpoint property with respect to $f$, then $G_1 +_f G_2 \in \mathscr{H}$.
\end{theorem}
\begin{proof}
 Let $G := G_1 +_f G_2$ with $f := uv$. Suppose $x \in \cone(\mathcal B (G)) \cap \lattice(\mathcal B (G))$. 
For $\gamma \in \mathbb{R}$ and $i \in \{1,2\}$, we define $x_i(\gamma) \in \mathbb{R}^{E(G_i)}$ as  $x_i(\gamma)_{e} := x_{e}$ if $e \neq f$ and $x_i(\gamma)_{f} := \gamma$. 
 
 \begin{claim}
For all $\gamma$, we have $x_1(\gamma) \in \lattice(\mathcal B (G_1))$ if and only if $x_2(\gamma) \in \lattice(\mathcal B (G_2))$. 
\end{claim}
\begin{subproof}
Suppose, without loss of generality, that $x_1(\gamma) \in \lattice(\mathcal B (G_1))$ and $x_2(\gamma) \notin \lattice(\mathcal B (G_2))$.  
Then $\gamma$ is an integer, and $G_2$ contains a circuit $C_2$ with $ x_2(\gamma)(C_2)$ odd.
Since $x \in \lattice(\mathcal B (G))$, we have $f \in C_2$.
Because $G_1$ is 2-connected, there is a circuit $C_1$ in $G_1$ with $f \in C_1$.
We have that $x_1(\gamma)(C_1)$ is an even integer, since $x_1(\gamma) \in \lattice(\mathcal B (G_1))$. Now $C_1 \triangle C_2$ is a circuit in $G$ with 
%Since $x_2(\gamma) \notin \lattice(\mathcal B (G_2))$, but $x \in \lattice(\mathcal B (G))$, it follows that there is a circuit $C_2$ of $G_2$ with $f \in C_2$ and $x_2(\gamma)(C_2)$ odd.  Since $G_1 \setminus f$ is connected, there is a circuit $C_1$ in $G_1$ with $f \in C_1$.  As $x_1 \in \lattice(\mathcal B (G_1))$, it follows that $x_1(C_1)$ even.  But now, $C_1 \triangle C_2$ is a circuit of $G$ with 
\[
x(C_1 \triangle C_2)=x_1(\gamma)(C_1)+x_2(\gamma)(C_2)-2\gamma \equiv 1 \pmod{2},
\]
contradicting $x \in \lattice(\mathcal B (G))$ and proving the claim.  
\end{subproof}  

\medskip

% Note that $x_i(\gamma)$ is almost in the lattice of $G_1$ and $G_2$ with respect to $f$ if $\gamma \in \mathbb{Z}$. Let
%\[ \gamma_{0} := \sum_{S \subseteq V(G) \; \vert \; \mid \{u,v\} \cap S \mid = 1 } \beta_S. \]
Since $x \in \cone(\mathcal B (G))$, there exist non-negative coefficients $\beta_S$ such that
 \[ x = \sum_{S \subseteq V(G)} {\beta_{S} \delta_G(S)}.  \]
Let $\mathcal{F}:=\{S \subseteq V(G) : |S \cap \{u,v\}|=1\}$ and define $\gamma' := \sum_{S \in \mathcal{F}} \beta_S.$  Consider the intervals $I_1 := I(G_1, x_1(\gamma'),f)$ and $I_2 := I(G_2, x_2(\gamma'),f)$.  Note that  
%$I_1 \cap I_2$ is non-empty since $\gamma' \in I_1 \cap I_2$. 
$\gamma' \in I_1 \cap I_2 \neq \emptyset$.

\begin{claim}
There exists $\gamma \in I_1 \cap I_2$ such that $x_1(\gamma) \in \lattice(\mathcal B (G_1))$ and $x_2(\gamma) \in \lattice(\mathcal B (G_2))$.  
\end{claim}
\begin{subproof}
%By Lemma \ref{lemma1}, for $i=1,2$, there exists $\gamma_i \in I_i$ such that $x_i(\gamma_i) \in \intcone(\mathcal B (G_i))$. 
 By Lemma \ref{lemma1}, there exists $\gamma_1 \in I_1$ such that $x_1(\gamma_1) \in \intcone(\mathcal B (G_1))$. 
If $I_1 \subseteq I_2$, we have $\gamma_1 \in I_2$.  
By the previous claim, $x_2(\gamma_1) \in \lattice(\mathcal B (G_2))$ and the claim is proved with $\gamma = \gamma_1$.
%Similarly, $I_2 \nsubseteq I_1$.  
Thus, we may assume that $I_1 \nsubseteq I_2$
and similarly, that $I_2 \nsubseteq I_1$.  
Since $I_1 \cap I_2 \neq \emptyset$, there exists a non-zero endpoint $\gamma_2$ of $I_1$ such that $\gamma_2 \in I_2$.  
As $G_1$ has the lattice endpoint property with respect to $f$, it follows that $x_1(\gamma_2) \in \lattice(\mathcal B (G_1))$.  
By the previous claim, $x_2(\gamma_2) \in \lattice(\mathcal B (G_2))$, so the claim is proved with $\gamma=\gamma_2$.  
\end{subproof}

%Suppose, with out loss of generality, that $x_1(\gamma) \in \lattice(\mathcal B (G_1))$ and $x_2(\gamma) \notin \lattice(\mathcal B (G_2))$.  
%It remains to show that $x_2(\gamma)(C_2)$ is an even integer for every circuit $C_2$ of $G_2$. 
%If $f \notin C_2$, then we are done since $x \in \lattice(\mathcal B (G))$. We may assume $f \in C_2$
%be a circuitThen $\gamma$ is an integer, and there exists a circuit $C_2$ of $G_2$ with $ x(\gamma)(C_2)$ odd.
%Since $x \in \lattice(\mathcal B (G))$ we have $f \in C_2$.
%Because $G_1$ is 2-connected, there exists a circuit $C_1$ in $G_1$ with $f \in C_1$.
%Since $x_1(\gamma) \in \lattice(\mathcal B (G_1))$, we have that $\gamma$ is an integer and $x_1(\gamma)(C_1)$ is an even integer.
%Now $C_1 \triangle C_2$ is a circuit in $G$ and  $x(C_1 \triangle C_2)$ is an even integer since $x \in \lattice(\mathcal B (G))$.
%We have
%\[
%x_2(\gamma)(C_2) = x(C_1 \triangle C_2) - x_2(\gamma)(C_2)+2\gamma \equiv 1 \pmod{2}
%\]
%Since $x_2(\gamma) \notin \lattice(\mathcal B (G_2))$, but $x \in \lattice(\mathcal B (G))$, it follows that there is a circuit $C_2$ of $G_2$ with $f \in C_2$ and $x_2(\gamma)(C_2)$ odd.  Since $G_1 \setminus f$ is connected, there is a circuit $C_1$ in $G_1$ with $f \in C_1$.  As $x_1 \in \lattice(\mathcal B (G_1))$, it follows that $x_1(C_1)$ even.  But now, $C_1 \triangle C_2$ is a circuit of $G$ with 
%\[
%x(C_1 \triangle C_2)=x_1(\gamma)(C_1)+x_2(\gamma)(C_2)-2\gamma \equiv 1 \pmod{2},
%\]
%contradicting $x \in \lattice(\mathcal B (G))$ and proving the claim.  

Let $\gamma$ be as above.  Since $G_i \in \mathscr{H}$, we have $x_i(\gamma) \in \intcone(\mathcal B (G_i))$. Thus, there exists a sequence $A_1, \dots, A_j$ of subsets of $V(G_1)$ and a sequence of subsets $B_1, \dots, B_k$ of $V(G_2)$ such that  $x_1(\gamma) = \sum_{i=1}^j \delta_{G_1}(A_i)$ and  $x_2(\gamma) = \sum_{i=1}^k \delta_{G_2}(B_i)$.   By taking complements if necessary, we may assume for each $i \in [j]$, $A_i \cap \{u,v\}=\emptyset$ or $A_i \cap \{u,v\}=\{u\}$ and that for all $i \in [k]$, $B_i \cap \{u,v\}=\emptyset$ or $B_i \cap \{u,v\}=\{u\}$.  Therefore, by re-indexing, we may assume $A_i \cap \{u,v\} = \{u\}$ for $i \in \{1, \dots, \gamma\}$ and $A_i \cap \{u,v\}=\emptyset$ for $i \in \{\gamma+1, \dots, j\}$.  Similarly, $B_i \cap \{u,v\} = \{u\}$ for $i \in \{1, \dots, \gamma\}$ and $B_i \cap \{u,v\}=\emptyset$ for $i \in \{\gamma+1, \dots, k\}$.  
Observe that $\delta_{G_1}(A_i) = \delta_G(A_i)$ for $i > \delta$ and $\delta_{G_2}(B_i) = \delta_G(B_i)$ for $i > \delta$.  Therefore, 
\[ x =  \sum_{i=1}^{\gamma} \delta_{G}(A_i \cup B_i) + \sum_{i=\gamma+1}^j \delta_G(A_i) + \sum_{i=\gamma+1}^{k} \delta_G(B_i). \]
So, $x \in \intcone(\mathcal B (G))$, as required.  
\end{proof}

Using Theorem~\ref{main}, we obtain the following connection between subdivision and deletion.

\begin{theorem}\label{delsubdiv}
Let $G \in \h$ and $f \in E(G)$. Then $G \setminus f \in \mathscr{H}$ if and only if $G +_{f} C_n \in \mathscr{H}$ for all $n \geq 4$.
\end{theorem}
\begin{proof} 
Note that $G +_{f} C_4$ is just the graph obtained from $G$ by subdividing $f$ twice.  We may assume that $G$ is 2-connected by Lemma~\ref{kcliquesums}.
The forward implication then follows directly from Theorem~\ref{main}.

For the converse, assume $G +_{f} C_n \in \mathscr{H}$ and let $u$ and $v$ be the ends of $f$. Now if $G \setminus f \notin \mathscr{H}$, then there exists \[x \in \cone(\mathcal{B}(G \setminus f)) \cap \lattice(\mathcal B (G \setminus f)) \setminus \intcone(\mathcal B (G \setminus f)).\]

Since $x \in \cone(G \setminus f)$, we have that $I(G, x, f) \neq \emptyset$.  Now, as $x \in \lattice (\mathcal B (G \setminus f))$, it follows that the parity of $x(P)$ is the same for all paths $P$ in $G \setminus f$ between $u$ and $v$.  Define $y \in \mathbb{Z}_{\geq 0}^{E(C_n) \setminus f}$ by setting $y_e := a$ for all $e$.  Let $y'$ be a vector obtained from $y$ by changing a single entry from $a$ to $a+1$.  Note that $I(C_n, y, f)=[0, (n-1)a]$, and since $n \geq 4$,  $I(C_n, y', f)=[0, (n-1)a + 1]$.  Let $z \in \mathbb{Z}_{\geq 0}^{E(G +_{f} C_n)}$ be the concatenation of $x$ and $y$ and let
$z' \in \mathbb{Z}_{\geq 0}^{E(G +_{f} C_n)}$ be the concatenation $x$ and $y'$.  

By choosing $a$ sufficiently large, we either  have $z \in \cone(\mathcal B (G +_f C_n)) \cap \lattice(\mathcal B (G +_f C_n)$ or $z' \in \cone(\mathcal B (G +_f C_n) \cap \lattice(\mathcal B (G +_f C_n))$.  However, neither $z$ nor $z'$ belong to $\intcone(\mathcal B (G +_f C_n))$, since when restricted to $E(G \setminus f)$, they are both equal to $x$.  This contradicts $G +_{f} C_n \in \mathscr{H}$.
\end{proof}

Note that one direction of the above proof breaks down if $f$ is subdivided only once, but the following conjecture may still be true.  

\begin{conjecture} 
Let $G \in \h$ and $f \in E(G)$. Then $G \setminus f \in \mathscr{H}$ if and only if $G +_{f} C_3 \in \mathscr{H}$.
\end{conjecture}

For our second application of Theorem~\ref{main}, we show that all $K_5^{\perp}$-minor-free graphs are in $\h$. 
First we require the following well-known lemma.  See~\cite{smallminor}, for a proof.

\begin{lemma} \label{small}
If $G$ is a 3-connected $K_5^{\perp}$-minor-free graph, then $G$ is $K_5$-minor-free or $G \cong K_5$.
\end{lemma}

\begin{theorem} \label{H6free}
All $K_5^{\perp}$-minor-free graphs are in $\mathscr{H}$. 
\end{theorem}

\begin{proof}
Let $G$ be a counterexample with $|V(G)|+|E(G)|$ minimum.  Hence, $G$ is simple and by Lemma~\ref{kcliquesums}, $G$ is also 2-connected.  If $G$ is 3-connected, then  by Lemma~\ref{small}, $G$ is $K_5$-minor-free or $G \cong K_5$.  In either case, $G \in \h$ by Lemma~\ref{K5in} or Lemma~\ref{noK5in}. Thus, $G=G_1 +_f G_2$ or $G=G_1 \oplus_f G_2$  for some $G_1$ and $G_2$ with a common edge $f$ and $|E(G_1)|, |E(G_2)| \geq 3$.  The latter is impossible by Lemma~\ref{kcliquesums}, so $G=G_1 +_f G_2$.   Among all possible such choices, choose $G_1$ and $G_2$ so that $|E(G_1)|$ is minimum.  Thus, $G_1$ is 3-connected or $G_1 \cong K_3$.  By Lemma~\ref{small}, $G_1$ is $K_5$-minor-free or $G_1 \cong K_5$.  In either case, $G_1$ has the lattice endpoint property with respect to $f$ by Lemma~\ref{k5LEP} or Lemma~\ref{noK5LEP}. Moreover, since $G$ is 2-connected, $G_1$ and $G_2$ are also 2-connected.  Finally,  by minimality, all four of the graphs $G_1, G_2$, $G_1 \setminus f$ and $G_2 \setminus f$ belong to $\h$.  Therefore, by Theorem~\ref{main}, $G \in \h$.
\end{proof}

\begin{remark}
It is also claimed in Laurent~\cite{laurent} that all $K_5^{\perp}$-minor-free graphs are in $\h$.  However, as far as we can see, the proof given (on page 260) assumes $\h$ is closed under \emph{$2$-sums}.  We now know that this is false in general by Corollary~\ref{2sum} (although it is true for \emph{2-clique sums}).  Therefore, we believe a different approach (such as the one above) is needed.
\end{remark}

\section{Open Problems}

Note that it is a bit of a curiosity that we do not know $\cone(G)$ explicitly, when $G$ is $K_5^{\perp}$-minor-free.  This appears to be a rare phenomenon.  Typically, it is necessary to know $\cone(\mathcal{B}(G))$ to show that $G \in \h$.  We thus have the following natural open problem.

\begin{problem}
Give an explicit description of $\cone(\mathcal{B}(G))$, when $G$ is $K_5^{\perp}$-minor-free.  
\end{problem}

Next, observe that by Theorem~\ref{H6free} and Theorem~\ref{k6minusedge}, all $K_5^{\perp}$-minor-free graphs are in $\h$, while all graphs with a $(K_6 \setminus e)$-minor are \emph{not} in $\h$.  There  are still many graphs that are not covered by these two theorems.  One such class of graphs are the uncontractions of $K_5$.  In~\cite{tanmay}, it is shown that all uncontractions of $K_5$ are in fact in $\h$.  However, the proof in~\cite{tanmay} is computer assisted.  Namely, there are 22 (non-isomorphic) 3-connected uncontractions of $K_5$, and it is verified by computer that each of these graphs is in $\h$.  
The general case then follows easily from Lemma~\ref{noK5in} and Theorem~\ref{delsubdiv}.  

It would be quite interesting to obtain this result without computer aid.

\begin{problem}
Give a human proof that all uncontractions of $K_5$ are in $\h$.
\end{problem}

Indeed, it turns out that $\cone(\mathcal{B}(G))$ has a very simple and beautiful description when $G$ is a 3-connected uncontraction of $K_5$, see~\cite[Theorem 2.3.3]{tanmay}.  Unfortunately, 
this characterization was also obtained by computer.  

\begin{problem}
Give a human proof that if $G$ is a 3-connected uncontraction of $K_5$, then the description of $\cone(\mathcal{B}(G))$ given in~\cite{tanmay}  is correct.
\end{problem}

Finally, by Corollary~\ref{edel}, the class $\h$ is not minor-closed.  Thus, $\h$ does not have a forbidden-minor characterization, but it may still be possible to give some alternate 
characterization of $\h$.

\begin{problem}
Characterize all graphs that are in $\h$.  
\end{problem}

\bibliography{references}{}

\begin{thebibliography}{10}

\bibitem{graphswithccp}
Brian Alspach, Luis Goddyn, and Cun~Quan Zhang.
\newblock Graphs with the circuit cover property.
\newblock {\em Trans. Amer. Math. Soc.}, 344(1):131--154, 1994.

\bibitem{AvisImaiIto}
David Avis, Hiroshi Imai, and Tsuyoshi Ito.
\newblock Generating facets for the cut polytope of a graph by triangular
  elimination.
\newblock {\em Math. Program.}, 112(2, Ser. A):303--325, 2008.

\bibitem{normaliz3}
Winfried Bruns and Bogdan Ichim.
\newblock Normaliz: algorithms for affine monoids and rational cones.
\newblock {\em J. Algebra}, 324(5):1098--1113, 2010.

\bibitem{normaliz4}
Winfried Bruns and Gesa K{\"a}mpf.
\newblock A {M}acaulay2 interface for {N}ormaliz.
\newblock {\em J. Softw. Algebra Geom.}, 2:15--19, 2010.

\bibitem{tanmay}
Tanmay Deshpande.
\newblock {Cones, lattices and {H}ilbert bases of cuts}.
\newblock Master's thesis, Simon Fraser University, 2013.

\bibitem{deza1}
M.~Deza.
\newblock On the hamming geometry of unitary cubes.
\newblock {\em Doklady Akademii Nauk SSR (in Russian) (resp. Soviet Physics
  Doklady (English translation))}, 134 (resp. 5), 1960 (resp. 1961).

\bibitem{deza2}
M.~Deza.
\newblock Small pentagonal spaces.
\newblock {\em Rendiconti del Seminario Nat. di Brescia}, 7:269--282, 1982.

\bibitem{cutconefacets}
Michel Deza and Monique Laurent.
\newblock Facets for the cut cone. {I}.
\newblock {\em Math. Programming}, 56(2, Ser. A):121--160, 1992.

\bibitem{smallminor}
Guoli Ding and Cheng Liu.
\newblock Excluding a small minor.
\newblock {\em Discrete Appl. Math.}, 161(3):355--368, 2013.

\bibitem{EdmondsMatchingPolytope}
Jack Edmonds.
\newblock Maximum matching and a polyhedron with {$0,1$}-vertices.
\newblock {\em J. Res. Nat. Bur. Standards Sect. B}, 69B:125--130, 1965.

\bibitem{circuitcover}
Xudong Fu and Luis~A. Goddyn.
\newblock Matroids with the circuit cover property.
\newblock {\em European J. Combin.}, 20(1):61--73, 1999.

\bibitem{tdi}
F.~R. Giles and W.~R. Pulleyblank.
\newblock Total dual integrality and integer polyhedra.
\newblock {\em Linear Algebra Appl.}, 25:191--196, 1979.

\bibitem{matchings}
Luis~A. Goddyn.
\newblock Cones, lattices and {H}ilbert bases of circuits and perfect
  matchings.
\newblock In {\em Graph structure theory ({S}eattle, {WA}, 1991)}, volume 147
  of {\em Contemp. Math.}, pages 419--439. Amer. Math. Soc., Providence, RI,
  1993.

\bibitem{Gordan}
P.~Gordan.
\newblock Ueber die {A}ufl\"osung linearer {G}leichungen mit reellen
  {C}oefficienten.
\newblock {\em Math. Ann.}, 6(1):23--28, 1873.

\bibitem{laurent}
Monique Laurent.
\newblock Hilbert bases of cuts.
\newblock {\em Discrete Math.}, 150(1-3):257--279, 1996.
\newblock Selected papers in honour of Paul Erd{\H{o}}s on the occasion of his
  80th birthday (Keszthely, 1993).

\bibitem{LovaszMatchingLattice}
L{\'a}szl{\'o} Lov{\'a}sz.
\newblock Matching structure and the matching lattice.
\newblock {\em J. Combin. Theory Ser. B}, 43(2):187--222, 1987.

\bibitem{OSheaSebo}
Edwin O'Shea and Andr{\'a}s Seb{\"o}.
\newblock Alternatives for testing total dual integrality.
\newblock {\em Math. Program.}, 132(1-2, Ser. A):57--78, 2012.

\bibitem{Schrijver}
Alexander Schrijver.
\newblock {\em Theory of linear and integer programming}.
\newblock Wiley-Interscience Series in Discrete Mathematics. John Wiley \& Sons
  Ltd., Chichester, 1986.
\newblock A Wiley-Interscience Publication.

\bibitem{multiflows}
P.~D. Seymour.
\newblock Matroids and multicommodity flows.
\newblock {\em European J. Combin.}, 2(3):257--290, 1981.

\end{thebibliography}
\bibliographystyle{plain}

\end{document}